\documentclass[11pt]{amsart}

\usepackage{amsmath, amssymb, amscd, newlfont}

\newcommand{\Irr}{{\mathrm{Irr}}}

\newcommand{\Ind}{{\mathrm{Ind}}}

\numberwithin{equation}{section}

\newtheorem{Prop}[equation]{Proposition}
\newtheorem{Lem}[equation]{Lemma}

\newtheorem{Thm}[equation] {Theorem}

\title
[Trace Paley--Wiener Theorem]
{A Remark on a Trace Paley--Wiener Theorem}
\author{Goran Mui\'c}
\address{Department of Mathematics, Faculty of Science, University of Zagreb,
Bijeni\v cka 30, 10000 Zagreb,
Croatia} 
 \email{gmuic@math.hr}
\thanks{The  author acknowledges Croatian Science Foundation grant no. 9364.}

\begin{document}
\maketitle 
\begin{abstract}
  In this paper we prove a version of a trace Paley--Wiener theorem
  for tempered representations of  a reductive $p$--adic group. This is applied to complete certain investigation of Shahidi on the proof that
  a Plancherel measure is invariant of a $L$--packet of discrete series. 
\end{abstract}

\section{Introduction}

Let $G$ be a reductive $p$--adic group.
Let $Rep(G)$ be the category of smooth admissible complex representations of $G$ of finite length, and let 
$R(G)$ be the corresponding Grothendieck group. We write $\Psi(G)$  (resp., $\Psi^u(G)$) for the group of (resp., unitary) 
unramified characters of $G$. The group $\Psi(G)$   has a structure of an algebraic variety (a complex tours). The corresponding algebra of regular functions 
$\mathbb C\left[\Psi(G)\right]$ is generated by evaluations on elements of $G$ as a $\mathbb C$--algebra. The subgroup  $\Psi^u(G)$ is Zariski dense in 
 $\Psi(G)$. We say that a complex function is regular on $\Psi^u(G)$ if it is a restriction 
of a regular function on $\Psi(G)$. We observe that the restriction map from $\mathbb C\left[\Psi(G)\right]$  into functions on
 $\Psi^u(G)$ is injective since $\Psi^u(G)$ is Zariski dense in 
 $\Psi(G)$. 

We fix a minimal parabolic subgroup $P_0$, its Levi decomposition  $P_0=M_0U_0$, and, as usual related to these choices,  we fix a set of 
standard parabolic subgroups $P=MU$, where $M_0\subset M$, $P=MP_0$. Since the standard parabolic subgroup is determined by the choice of Levi  subgroup,
the normalized parabolic induction $\Ind_P^G(\sigma)$, where $\sigma$ is a smooth representation of $M$, we write as usual $i_{GM}(\sigma)$.

 In \cite{bdk}, Bernstein, Deligne, and Kazhdan  proved a trace Paley--Wiener theorem for category $Rep(G)$. We consider 
 a full subcategory $Rep_t(G)$ of $Rep(G)$ consisting of representations having all irreducible subquotients tempered. Let $R_t(G)$ be the corresponding Grothendieck group. 
We write  $R^i_t(G)$ for the subgroup of $R_t(G)$ generated by $i_{GM}(\sigma)$, where $M$ ranges over all  standard Levi subgroups of $G$ (including $G$), and 
$\sigma$ ranges over a set of square--integrable modulo center irreducible representations of $M$. We warn the reader that this notion is  
not an analogue of the notion of strictly induced modules from (\cite{bdk}, 3.1). An analogue would be the subgroup of $R_t(G)$ generated by $i_{GM}(\tau)$, where
where $M$ ranges over all {\bf proper} standard Levi subgroups of $G$, and  $\tau$ ranges over irreducible tempered representations of $M$. But this is not useful for us
in the present paper.

\vskip .2in 
The main result of the present paper is the following version of a trace Paley--Wiener theorem:

\begin{Thm}\label{main} Let $f: R_t(G)\longrightarrow \mathbb C$ be a $\mathbb Z$--linear form such that
  the following hold:
  \begin{itemize}
  \item[(i)] There exists an open compact subgroup $K\subset G$
which dominates $f$ (i.e., $f$ is non--zero only on those irreducible tempered representations which have a non--trivial space of 
$K$--invariant vectors),
\item[(ii)]  For each  standard maximal Levi subgroup $M$, or $M=G$,  and a 
square--integrable modulo center representation  $\sigma$ of $M$,  the  function $\psi\mapsto f(i_{GM}(\psi\sigma))$ is regular on 
$\Psi^u(M)$, and for any other proper standard Levi subgroup $N$, and a  square--integrable modulo center representation  $\tau$ of $N$,
  we have  $f(i_{GN}(\tau))=0$.
\end{itemize}
Then, there exists $F\in C_c^\infty(G)$ such that
  $$
  f(\pi)=\text{tr}{(\pi(F))}, \ \ \text{ for all $\pi\in R^i_t(G)$.}
  $$
\end{Thm} 

\vskip .2in
Theorem \ref{main}  is a proved by reduction to the main result of \cite{bdk}
using Harish-Chandra theory of tempered representations \cite{w} and some standard considerations related to the Langlands
classification (\cite{renard}, Chapter VII).   The proof is given in Section \ref{proof}. It is a consequence of its 
 effective version given by Proposition \ref{p}. Proposition \ref{p} constructs a correct function needed in the proof of (\cite{Sh}, Proposition 9.3 2))
in the case when $M$ (see notation there) is a Levi of a maximal parabolic subgroup. We remark that since Plancherel factors are multiplicative, it is enough to prove
(\cite{Sh}, Proposition 9.3 2)) for a maximal Levi subgroup. 

\vskip .2in
I would like to than Gordan Savin for turning my attention to this question. A draft of the paper was written while the author visited the Hong Kong University of Science and Technology
in January of 2018. The author would like to thank A. Moy and the  Hong Kong University of Science and Technology for their hospitality.

\section{Preliminaries} 

We continue with the notation introduced in the introduction. Let $M$ be a standard Levi subgroup. Then, we write 
$\Psi(M)^r$ for the group of all unramified characters $\psi$ which are $\mathbb R_{>0}$--valued. As we stated in the introduction, every standard
Levi subgroup $M$ determine unique standard parabolic subgroup, say $P$. We denote by $\Psi(M)^{r, +}$  the set of all characters from $\Psi(M)^r$
which correspond to the points of the (open) Weyl chamber determined by the roots of  the split component of $M$ which belong to the unipotent radical of $P$
in the usual description of unramified characters (see for example \cite{muic}, Section 2). If $M=G$, then  $\Psi(M)^{r, +}= \Psi(M)^r$.

For  a standard Levi subgroup $M$, an irreducible tempered 
representation $\pi$ of $M$, and  $\psi\in \Psi(M)^{r, +}$, the module $i_{GM}(\psi\pi)$ is called a 
standard module; it has a unique  (Langlands quotient) $L\left(i_{GM}(\psi\pi)\right)$. The condition is empty if $M=G$. 
By the Langlands classification (\cite{renard}, Theorem VII.4.2), every irreducible representation can be expressed in the form  $L\left(i_{GM}(\psi\pi)\right)$ for unique such 
datum $(M, \pi, \psi)$.  
The following standard result will be used in the proof:

\begin{Lem}\label{l-1} The standard modules of $G$ form a $\mathbb Z$--basis of $R(G)$.
\end{Lem}
\begin{proof} As in (\cite{clo}, Proposition 1).
\end{proof}

\vskip .2in 
In analogy with (\cite{bdk}, 2.1), we make the following definitions. 

\vskip .2in
Let $\sigma \in\Irr (M)$ where $M$ is a standard Levi subgroup of $G$.  We define the usual affine  variety attached to $\sigma$ 
$$
\Irr(M)\supset D(\sigma)=\Psi(M)\sigma=
\Psi(M)/Stab_{\Psi(M)}(\sigma),
$$
where $Stab_{\Psi(M)}(\sigma)$ is a finite group consisting of all
$\psi\in \Psi(M)$ such that  $\psi\sigma\simeq \sigma$. 

\vskip .2in
If $A$ is a maximal split torus in the centre of $M$, then the restriction map $\Psi(M)\rightarrow  \Psi(A)$ is surjective, and the kernel is a finite 
group. Therefore, by considering the restriction to $A$ we find that

$$
Stab_{\Psi^u(M)}(\sigma)=Stab_{\Psi(M)}(\sigma).
$$
So, we may consider
$$
D^u(\sigma)\overset{def}{=}\Psi^u(M)/Stab_{\Psi^u(M)}(\sigma)\subset D(\sigma).
$$
It is easy to see that $D^u(\sigma)$ is Zariski dense in $D(\sigma)$. 

\vskip .2in

The action of the Weyl group
$$
W(M)=N_G(M)/M
$$
on $\Psi(M)$ is algebraic. Moreover,  $w\in W(M)$ transforms $Stab_{\Psi(M)}(\sigma)$ onto $Stab_{\Psi(M)}(w(\sigma))$. 
So that it maps $D(\sigma)$ (resp., $D^u(\sigma)$) onto $D(w(\sigma))$ (resp., $D^u(w(\sigma))$).

\vskip .2in
Put $D=D(\sigma)$ and $D^u=D^u(\sigma)$. As usually, we consider the group  $W(D)$ of all $w\in W(M)$
 such that there exists 
$\psi_w\in \Psi(M)$ such that 
\begin{equation}\label{20000}
w(\sigma)\simeq \psi_w \sigma.
\end{equation}
The character $\psi_w$ is determined uniquely modulo $Stab_{\Psi(M)}(\sigma)$. The group $W(D)$ acts on
the affine variety $D=\Psi(M)/Stab_{\Psi(M)}(\sigma)$ as follows:

\begin{equation}\label{20001}
w.\psi Stab_{\Psi(M)}(\sigma)=\psi_w w(\psi) Stab_{\Psi(M)}(\sigma).
\end{equation}

\vskip .2in 
The resulting orbit space 
$$
D/W(D)
$$
is again an affine variety with algebra of regular functions given as usual
$$
\mathbb C\left[D/W(D)\right] =\mathbb C\left[D\right]^{W(D)}.
$$

\vskip .2in 

One can construct a regular function $D/W(D)$ in the following way:

\begin{Lem}\label{l-2} Let $F\in C_c^\infty(G)$. Then, the function
  $\psi\longmapsto tr\left(i_{GM}(\psi\sigma)(F)\right)$ is a regular function on
  $D/W(D)$.
\end{Lem}
\begin{proof} It is standard that this function is regular on $D$. We show that it is $W(D)$--invariant. 
Let $w\in W(D)$. By (\cite{bdk}, Lemma 5.4 (iii)), we have 
$$
 tr\left(i_{GM}(\psi\sigma)(F)\right)= tr\left(i_{GM}(w(\psi\sigma))(F)\right)
$$
which completes the proof.
\end{proof}

\vskip .2in 

The above explicit description shows that  analogously defined group $W(D^u)$ is a subgroup of 
$W(D)$. In fact, we have the following lemma:

\begin{Lem}\label{l-200} Assume that the central character $\omega_\sigma: A \longrightarrow \mathbb C^\times$ of $\sigma$ is unitary. Then,
  $W(D^u)=W(D)$. Moreover,   $D^u/W(D)$ is Zariski dense in $D/W(D)$.
\end{Lem}
\begin{proof} As we remarked above, it is always $W(D^u)\subset W(D)$. Conversely, if $w\in W(D)$, then $w(\sigma)\simeq \psi_w \sigma$ by (\ref{20000}).
  Considering central characters, we find that
  $$
  \omega_{w(\sigma)}= \left(\psi_w|_A\right) \omega_\sigma.
  $$
  This implies that $\psi_w|_A$ is an unitary character. By the standard description of unramified characters of $M$, and its relation to unramified characters of $A$, this implies that
  $\psi_w\in \Psi^u(M)$ (see \cite{muic}, Section 2). Hence, $w\in W(D^u)$.  This completes the proof that  $W(D^u)=W(D)$. The remaining claim is obvious from above considerations. 
  \end{proof}

\vskip .2in
The following lemma is a fundamental result of Harish--Chandra:

\begin{Lem}\label{l-3} Assume that $M$ and $N$ are standard Levi subgroups of $G$, and  $\sigma$ and $\tau$ are  square--integrable modulo center 
representations of $M$ and $N$, respectively.
Then,  $i_{GM}(\sigma)$ and  $i_{GN}(\tau)$ have a common irreducible subrepresentation if and only 
if there exists $w\in G$ such that $N=wMw^{-1}$ and $\tau\simeq w(\sigma)$, where $w(\sigma)$ is defined by
$w(\sigma)(n)=\sigma(w^{-1}nw),$ $n\in N$. Moreover, if there exists $w\in G$ such that $N=wMw^{-1}$, then 
 $i_{GM}(\sigma)$ and  $i_{GM}(w(\sigma))$ are isomorphic, and in particular equal in  $R_t(G)$. 
\end{Lem}
\begin{proof} \cite{w}.
\end{proof}

\vskip .2in
Motivated by (\cite{bdk}, 2.1), we proceed as follows.
By the standard theory of tempered irreducible representations due to Harish--Chandra (see \cite{w}), for an irreducible tempered representation  $\pi\in \Irr(G)$, 
there exists a standard Levi subgroup  $M$ and a square--integrable modulo center representation $\sigma$ of $M$ such that 
$\pi\hookrightarrow i_{GM}(\sigma)$.  The pair is $(M, \sigma)$ is unique up to a conjugation (see Lemma \ref{l-3}). 
We call the equivalence class   $[M, \sigma]$ under conjugation of the pair  $(M, \sigma)$ the $t$--infinitesimal character
of $\pi$. The set of equivalence of such pairs we denote by $\Theta_t(G)$.

\vskip .2in
For a pair $(M, \sigma)$, we define a natural map $\Psi^u(M)\longrightarrow \Theta_t(G)$ given by 
$$
\psi\longmapsto [M, \psi \sigma].
$$
The image is called a connected component of $\Theta_t(G)$. We denote it by  $\Theta_t(M, \sigma)$. This map induces a bijection
which enable us to identify
$$
\Theta_t(M, \sigma)=D^u(\sigma)/W(D(\sigma)).
$$

\vskip .2in
Thus, in view of Lemma \ref{l-200},   we may consider   
$$
\Theta_t(M, \sigma)\subset  
D(\sigma)/W(D(\sigma)).
$$
This realizes $\Theta_t(M, \sigma)$ as a Zariski dense subset of the affine variety 
$D(\sigma)/W(D(\sigma))$.

\vskip .2in
As in (\cite{bdk}, 2.1), we can decompose 

\begin{equation}\label{3000}
R_t(G)=\oplus_{\theta} R_t(G)(\theta),
\end{equation}
where $\theta$ ranges over connected components of $\Theta_t(G)$. Here 
$$
R_t(G)(\theta)
$$
is generated with all tempered irreducible representations which $t$--infinitesimal characters belong to $\theta$. We denote by $1_\theta$ the projector
$$
R_t(G)\rightarrow  R_t(G)(\theta),
$$
for all $\theta \in \Theta_t(G)$.

\vskip .2in
We end this section with an analogue for $Rep_t(G)$ of the Decomposition theorem    for  the category of all smooth complex representations of $G$
(see \cite{bdk}, 2.3; \cite{Be}, 2.10).

\begin{Lem}\label{l-4} Let $K\subset G$ be an open compact subgroup. Then, there exists a finite set $T_K$ consisting of connected components in $\Theta_t(G)$ such that for each irreducible tempered
  representation $\pi\in Rep_t(G)$, having non--zero space of $K$--invariants, there exists $\theta\in T_k$ such that $\pi \in R_t(G)(\theta)$.
\end{Lem}
\begin{proof} By the Decomposition theorem (see \cite{bdk}, 2.3), there exists a finite set, say $S$, of pairs  $(N,\  \rho)$, where $N$ is a  standard Levi subgroup of $G$, and
 $\rho$ are irreducible supercuspidal representations, such that for every irreducible representation $\pi$ of $G$, having non--zero space of $K$--invariants, there exists   $(N, \rho)\in S$, and
 an unramified character  $\chi$ such that $\pi$ is a subqoutient of $i_{G,N}\left(\chi \rho\right)$.

 Now, assume that $\pi$ is as in the statement of the lemma. Then, there exists a standard Levi subgroup  $M$ and a square--integrable modulo center $\sigma$ of $M$ such that 
 $\pi\hookrightarrow i_{GM}(\sigma)$.  Moreover, there exists   a standard Levi subgroup $M'$ of $M$ (and of $G$), and a supercuspidal irreducible representation $\rho'$ such that $\sigma$ is an
 irreducible subquotient of $i_{M, M'}\left(\rho'\right)$. By induction in stages, $\pi$ must be a subquotient of $i_{G, M'}\left(\rho'\right)$. By standard theory of induced representation \cite{BZ1},
 the pair $(M', \rho')$ must  be $G$--conjugate to the one in $S$.  Thus, we may assume that $(M', \rho')\in S$ already. 

 Thus, it is enough to prove that  given $(N,\  \rho)\in S$ and given  standard Levi subgroup $M$ of $G$ such that $N\subset M$, there are finitely many of $\Psi^u(M)$--orbits of square--integrable
 modulo center  representations    of $M$ such they are subqoutients of the induced representations in the family  $i_{M, N}\left(\chi \rho\right)$ parameterized by $\chi\in \Psi(N)$.
 But that is easy. We can select sufficiently small open compact
 subgroup $L\subset M$ such that every irreducible representation  that appear as a subqoutient of  $i_{M, N}\left(\chi \rho\right)$ for some $\chi\in \Psi(N)$ has a non--zero space of $L$--invariants.

 Hence, we need to prove that there  are finitely many of $\Psi^u(M)$--orbits of square--integrable
 modulo center  representations    of $M$ having  a non--zero space of $L$--invariants. This is proved in \cite{w} (see (iii) in the introduction of \cite{w}). 
  \end{proof}

\section{Proof of Theorem \ref{main}}\label{proof}

We begin the proof of Theorem \ref{main} with the following lemma:

\begin{Lem}\label{p-0} Let $f$ be as in the statement of Theorem \ref{main}. Then, there exists a finite set $T_f$ consisting of connected components in $\Theta_t(G)$ such that for each irreducible tempered
  representation $\pi\in Rep_t(G)$ such that $f(\pi)\neq 0$ there exists $\theta\in T_f$ such that $\pi \in R_t(G)(\theta)$.
\end{Lem}
\begin{proof} This follows from the assumption (i) in Theorem \ref{main} combined with Lemma \ref{l-4}. 
  \end{proof}

  \vskip .2in 
  By Lemma \ref{p-0}, we can decompose $f$ into  $\mathbb Z$--linear forms $f_\Theta: R_t(G)\longrightarrow \mathbb C$, $\theta\in  T_f$,
  $$
  f=\sum_{\theta\in  T_f} \ f_\theta,
  $$
  where $f_\theta$ is defined as follows (see (\ref{3000})):
  $$
  f_\theta=f\circ 1_\theta.
  $$
  Obviously, each $f_\theta$ satisfies the assumptions analogous to (i) and (ii) in Theorem \ref{main}.

\vskip .2in 
Hence, in what follows we may assume   that $f=f_\theta$ for some $\theta \in \Theta_t(G)$. By the assumption (ii) of Theorem \ref{main}, we may assume that $\theta$
has  the form $\theta=\Theta_t(M, \sigma)$,  where $M$ is a standard maximal Levi subgroup of $G$, or $M=G$, and $\sigma$ is a $\sigma$ is a
square--integrable modulo center representation of $M$. We observe that 

$$
 \psi\in \Psi^u(M) \longmapsto  f\left(i_{GM}(\psi\sigma)\right) 
 $$
 is a regular function by the assumption (ii) of Theorem \ref{main}. Thus, by definition this means, that it is restriction of a regular function, say $a$
 of on the affine variety $\Psi(M)$. By  Lemma \ref{l-3}, we have

 \begin{equation}\label{p-01}
 a\in \mathbb C\left[D\right]^{W(D)},
\end{equation}
 where
 \begin{equation}\label{p-02}
 D=\Psi(M)/Stab_{\Psi(M)}(\sigma).
 \end{equation}
 We refer to previous section for the notation.

 \vskip .2in
Now, the following proposition completes the proof of the theorem. 
 
\begin{Prop}\label{p} Let $M$ be a standard maximal Levi subgroup of $G$, or $M=G$. Assume that $\sigma$ is a square--integrable modulo center representation of $M$. We define
  $D$ by (\ref{p-02}), and let $a$ be {\bf any} function in $\mathbb C\left[D\right]^{W(D)}$. Then, 
there exists $F\in C_c^\infty(G)$ such that 
$$
tr{(\pi(F))}
=\begin{cases}
a(\psi) \ \ \text{for} \ \ \pi=i_{GM}(\psi\sigma), \ \ \psi\in \Psi^u(M),\\
0 \ \ \  \ \text{for} \ \ \pi=i_{GN}(\psi\tau), \ \ \psi\in \Psi^u(N),\\
\end{cases}
$$
for any other standard  Levi subgroup $N$ and a square--integrable modulo center represenrepresentation $\tau$ such that
$\Theta_t(N, \tau)\neq \Theta_t(M, \sigma)$.
\end{Prop}
\begin{proof} The proof of Proposition \ref{p} is a generalization of  (\cite{clo}, 4.2, Proposition 1) where the proof of existence of pseudo--coefficients 
  for semisimple $G$ is given based also on \cite{bdk}.  We consider only the case $M$ is a standard maximal Levi subgroup of $G$.
  The case of $M=G$ is about the construction of a specific pseudo--coefficient of $\sigma$. The proof is on the same lines but considerably easier.

We remark that $\Psi^u(G)$ acts on  $\Psi^u(M)$ in a usual way:
$$
\psi\longmapsto \chi|_M \psi, \ \  \chi \in \Psi^u(G), \ \ \psi\in \Psi^u(M).
$$
For $\psi\in \Psi^u(M)$, the stabilizer 
$$
Stab_{\Psi^u(G)}(i_{GM}(\psi\sigma))
$$
is the group of all $\chi \in \Psi^u(G)$ such that 
$$
\chi i_{GM}(\psi\sigma)\simeq  i_{GM}(\psi\sigma).
$$
We remind the reader that for all $\chi \in \Psi^u(G)$ we have

$$
\chi i_{GM}(\psi\sigma)\simeq  i_{GM}(\chi|_M \psi\sigma).
$$

\vskip .2in 
\begin{Lem}\label{p-2} 
 Assume that $\chi \in \Psi^u(G)$  and $\psi\in \Psi^u(M)$.
Then, for each irreducible constituent $\pi$ of $i_{GM}(\psi\sigma)$, the multiplicity of 
$\chi\pi$ in $\chi i_{GM}(\psi\sigma)$ is same as that of $\pi$  in $i_{GM}(\psi\sigma)$.
\end{Lem}
\begin{proof} Obvious.
\end{proof}

\vskip .2in
\begin{Lem}\label{p-3} 
Assume that for $\chi \in \Psi^u(G)$  and $\psi\in \Psi^u(M)$
there exists an irreducible constituent $\pi$ of $i_{GM}(\psi\sigma)$ such that $\chi\pi$ is an irreducible constituent of $i_{GM}(\psi\sigma)$.
Then,  $\chi\in Stab_{\Psi^u(G)}(i_{GM}(\psi\sigma))$. In particular, we have
$$
Stab_{\Psi^u(G)}(\pi)\subset  Stab_{\Psi^u(G)}(i_{GM}(\psi \sigma)).
$$
\end{Lem}
\begin{proof} First, 
$\chi\pi$ is a common constituent of $i_{GM}(\psi\sigma)$ and $i_{GM}(\chi|_M \psi\sigma)$. So, by Lemma \ref{l-3}, there exists $w\in W(M)$ 
such that 
$$
\chi|_M\psi\sigma=w(\psi\sigma).
$$
Then, again  by Lemma \ref{l-3}, we obtain
$$
\chi i_{GM}(\psi\sigma) \simeq  i_{GM}(\chi|_M \psi\sigma) \simeq  i_{GM}(\psi\sigma).
$$
\end{proof}

\vskip .2in
\begin{Lem}\label{p-4} Let $\psi\in \Psi^u(M)$. Then, we have the following:
\begin{itemize}
\item[(i)] If $\chi\in Stab_{\Psi^u(G)}(i_{GM}(\psi\sigma))$, then 
$a(\chi|_M \psi)=a(\psi)$.
\item[(ii)] For each $\eta\in \Psi(G)$ and  $\chi\in Stab_{\Psi^u(G)}(i_{GM}(\psi\sigma))$, we have
$$
a(\chi|_M \eta|_M \psi)=a(\eta|_M \psi).
$$
\end{itemize}
\end{Lem}
\begin{proof}  We prove (i). Since $\chi\in Stab_{\Psi^u(G)}(i_{GM}(\psi\sigma))$, we obtain
$$
 i_{GM}(\chi|_M \psi\sigma) \simeq \chi i_{GM}(\psi\sigma)\simeq  i_{GM}(\psi\sigma).
$$
So,  by Lemma \ref{l-3}, there exists $w\in W(M)$ 
such that 
$$
\chi|_M\psi\sigma\simeq w(\psi\sigma) \simeq w(\psi) w(\sigma).
$$

By definition of $W(D)$ (see (\ref{20000})), this implies 
$w\in W(D)$, and above relation can be written as follows:
$$
\chi|_M\psi\sigma\simeq \psi_w w(\psi)\sigma,
$$
where
$$
\psi_w=  w(\psi)^{-1} \chi|_M\psi.
$$
Consequently, by the definition of the action of $W(D)$ on $D$ (see (\ref{20001})) we obtain
$$
\chi|_M\psi  Stab_{\Psi(M)}(\sigma)= \psi_w w(\psi) Stab_{\Psi(M)}(\sigma)= w.\psi Stab_{\Psi(M)}(\sigma).
$$ 
This implies $a(\chi|_M \psi)=a(\psi)$. This proves (i).

To prove (ii), we may assume that  $\eta$ is unitary. Then, we obviously have
$$
Stab_{\Psi^u(G)}(i_{GM}(\eta|_M\psi \sigma))=Stab_{\Psi^u(G)}(i_{GM}(\psi \sigma)).
$$
Now, the claim follows from (i).
\end{proof}

\vskip .2in

Now, in order to complete the proof of Proposition \ref{p}, we apply (\cite{bdk}, Theorem 1.2). We define a $\mathbb Z$--linear form $f: R(G)\longrightarrow \mathbb C$ in
several steps.  We warn the reader that we use the same letter for a functional different than one from the statement of Theorem \ref{main}.

\vskip .2in
\noindent{1.} For each $\Psi^u(G)$--orbit $\cal O$ in $\Psi^u(M)$, we fix a representative $\psi_{\cal O}\in \cal O$ and an irreducible 
constituent $\pi_{\cal O}$ in $i_{GM}(\psi_{\cal O}\sigma)$. By Lemma \ref{p-3}, we have
\begin{equation}\label{p-4a}
Stab_{\Psi^u(G)}(\pi_{\cal O})\subset  Stab_{\Psi^u(G)}(i_{GM}(\psi_{\cal O}\sigma)).
\end{equation}
The quotient is finite and if $\chi$ ranges over representatives of
the quotient, then $\chi \pi_{\cal O}$ ranges over the set of all mutually non--equivalent irreducible subrepresentations in 
$i_{GM}(\psi_{\cal O} \sigma)$ which are $\Psi^u(G)$--equivalent to $\pi_{\cal O}$. Any of those representations, have the same multiplicity 
in  $i_{GM}(\psi_{\cal O} \sigma)$. Let $m_{\cal O}$ be the sum of their multiplicities. We define:

$$
f\left(\chi\pi_{\cal O}\right)=\frac{a(\psi_{\cal O})}{m_{\cal O}}, \ \  \chi\in Stab_{\Psi^u(G)}(i_{GM}(\psi_{\cal O}\sigma)).
$$

\vskip .2in
\noindent{2.} For each $\chi \in \Psi^u(G)$, we obviously have 
$$
Stab_{\Psi^u(G)}(\chi \pi_{\cal O})=Stab_{\Psi^u(G)}(\pi_{\cal O})
$$
and 
$$
Stab_{\Psi^u(G)}(i_{GM}(\chi|_M\psi_{\cal O}\sigma))=Stab_{\Psi^u(G)}(i_{GM}(\psi_{\cal O}\sigma)).
$$
By, Lemma \ref{p-2} and these remarks, the sum of multiplicities of  $\Psi^u(G)$--equivalent representations of $\pi_{\cal O}$
which belong to $i_{GM}(\chi|_M\psi_{\cal O}\sigma)$ is again $m_{\cal O}$. We let 
$$
f\left(\chi\pi_{\cal O}\right)=\frac{a(\chi|_M \psi_{\cal O})}{m_{\cal O}}, \ \  \chi \in \Psi^u(G).
$$
Lemma \ref{p-4} (ii) shows that this is well--defined.

\vskip .2in
\noindent{3.} For any other tempered irreducible representation (and, in particular, square--integrable modulo center representation) $\pi$ of $G$ we let 
$$ 
f(\pi)=0.
$$

\vskip .2in
\noindent{4.} For any quasi--tempered irreducible representation $\pi$ of $G$, we can write $\pi=\chi\pi^u$, where $\chi \in \Psi^r(G)$ and $\pi^u$ is tempered.  We let 
$$ 
f(\pi)=0,
$$
if $\pi^u$ is not in $\Psi^u(G)\pi_{\cal O}$ for any orbit $\cal O$ described in 1. But, if   $\pi^u\in \Psi^u(G)\pi_{\cal O}$, for some $\cal O$, then
we can write  $\pi^u= \psi \pi_{\cal O}$, for some $\psi \in \Psi^u(G)$ uniquely determined modulo  $Stab_{\Psi^u(G)}(\pi_{\cal O})$. We let
$$
f\left(\pi\right)=\frac{a(\chi|_M  \psi|_M  \psi_{\cal O})}{m_{\cal O}}.
$$
Using (\ref{p-4a}) and Lemma \ref{p-4} (ii) we see that this is well--defined.

\vskip .2in
\noindent{6.} Finally, we define $f$ on non--tempered Langlands quotients (see Lemma \ref{l-1}).  let  $f$ to be  equal to zero on all standard modules 
induced from proper parabolic subgroups except in the following two obvious cases:

\begin{itemize}
  \item[(a)] The standard module  $i_{GM}(\chi \psi\sigma)$, 
where $\chi \in \Psi(M)^{r, +}$ and  $\psi\in \Psi^u(M)$. In this case, we let 
$$
f\left(i_{GM}(\chi \psi\sigma)\right)=a(\chi \psi).
$$

\item[(b)] It is also possible that $\chi \in \Psi(M)^r$ belongs to the positive Weyl Chamber for the opposite parabolic $\overline{P}$ (see beginning of previous section).
  Then, there exists a unique standard maximal parabolic subgroup $Q$ with standard Levi $N$, and $w\in G$ such that $N=wMw^{-1}$. Now, by (\cite{bdk}, Lemma 5.3 (iii)),  we have
$$
i_{GM}(\chi \psi\sigma)=i_{GN}(w(\chi) w(\psi)w(\sigma)),
$$
in $R(G)$. Also, we have  $w(\chi)\in  \Psi(N)^{r, +}$. On the standard module $i_{GN}(w(\chi) w(\psi)w(\sigma))$ we let
$$
f\left(i_{GN}(w(\chi) w(\psi)w(\sigma))\right)= a(\chi \psi).
$$
Thus, we have
$$
f\left(i_{GM}(\chi \psi\sigma)\right) = f\left(i_{GN}(w(\chi) w(\psi)w(\sigma))\right)= a(\chi \psi),
$$
for $\chi \in \Psi(M)^r$ such that  $w(\chi)\in  \Psi(N)^{r, +}$.
\end{itemize}

The third case is that $\chi \in \Psi(M)^r$ is in neither chamber. Then,  $\chi \in \Psi(G)^r$, by standard description of unramified characters (\cite{muic}, Section 2). In this case  
$$
i_{GM}(\chi \psi\sigma)=\chi i_{GM}(\psi\sigma)
$$
is a quasi--tempered representation, and, by
$$
f\left(i_{GM}(\chi \psi\sigma)\right)=f\left(\chi i_{GM}( \psi\sigma)\right)= a(\chi \psi),
$$
by 1.--4.

\vskip .2in 

This  completes the construction of  $\mathbb Z$--linear form $f: R(G)\longrightarrow \mathbb C$.  In order to complete the proof of 
the proof of Proposition \ref{p}, we just need to check that it satisfies the assumptions of (\cite{bdk}, Theorem 1.2). First, let $N$ be a standard Levi subgroup of $G$ contained in $M$,
and $\rho$ an irreducible supercuspidal representation of $N$ such that $\sigma$ is an irreducible subqoutient of  $i_{M, N}\left(\rho\right)$. Then, by construction, $f$ is zero on
irreducible representations which are not irreducible subqoutients of members   of the family  $i_{M, N}\left(\chi \rho\right)$ parameterized by $\chi\in \Psi(N)$. Then, as in the
proof of Lemma \ref{l-4}, there exists an open compact subgroup $K$ such that $f$ is zero on all irreducible representations which does not have a non--zero $K$--invariant vector.
This is (ii) in (\cite{bdk}, 1.2). It remains to check  (i) in (\cite{bdk}, 1.2). We need to check that for an arbitrary standard Levi subgroup $N$ of $G$ and an irreducible representation
$\tau$ of $N$, the function $\chi\longmapsto f\left(i_{G, N}(\chi \tau)\right)$ is regular on $\Psi(N)$. By Lemma \ref{l-1} applied to $N$, $\tau$ is a $\mathbb Z$--linear combination  
of standard modules for $N$. So,  instead of being irreducible, we may assume that
$\tau$ is a standard module for $N$ i.e.,
$$
\tau=i_{NN'}\left(\chi'\tau'\right),
$$
$N'$ is a standard Levi subgroup, $\tau'$ is an irreducible tempered representation of $N'$ and $\chi'\in \Psi^{r, +}(N', N)$. Here, by definition $\Psi^{r, +}(N', N)$ is an analogue of
$\Psi^{r, +}(N', G)\overset{def}{=}\Psi^{r, +}(N')$ defined in previous section. Now, by induction in stages, we have
$$
i_{G, N}\left(\chi \tau\right)=i_{G, N'}\left(\chi|_{N'}\chi' \tau'\right).
$$
We decompose $\chi=\chi^r\chi^u$ into its real $\chi^r\in \Psi^r(N)$   and unitary part $\chi^u\in \Psi^r(N)$. Let $N{''}$ be a standard Levi subgroup such that
$N'\subset N{''}\subset N$ obtained by adjoining all simple roots orthogonal to $\chi^r|_{N'}\chi'$ (see \cite{muic}, Section 2). Then, $\chi^r|_{N'}\chi'$  is an unramified character of
$N{''}$ which is not orthogonal to any simple root that determines standard parabolic subgroup of $N{''}$. In particular, there exists $w\in G$ such that
$N_1{''}=wN{''}w^{-1}$ is a standard Levi subgroup, and  
$$
w\left(\chi^r|_{N'}\chi' \right)\in \Psi^{r, +}(N_1{''})
$$
(see for example \cite{muic-0}, Section 1). 
Also, we can write
$$
i_{G, N'}\left(\chi|_{N'}\chi' \tau'\right)= i_{G, N{''}}\left(\chi^r|_{N'}\chi' \ i_{N', N{''}}\left(\chi^u|_{N'}\tau'\right)\right).
$$
Obviously, $i_{N', N{''}}\left(\chi^u|_{N'}\tau'\right)$ is a direct sum of irreducible tempered representations, say $\tau{''}$ of $N{''}$. This implies that
$i_{G, N'}\left(\chi|_{N'}\chi' \tau'\right)$ is a direct sum induced representations:
$$
i_{G, N{''}}\left(\chi|_{N'}\chi' \tau{''}\right).
$$
By above, in $R(G)$, we have

\begin{equation}\label{eee}
i_{G, N{''}}\left(\chi|_{N'}\chi' \tau{''}\right)=i_{G, N_1{''}}\left(w\left(\chi|_{N'}\chi'\right) w\left(\tau{''}\right)\right).
\end{equation}
But the last induced representation is a standard module. Now, by the construction of $f$, $f=0$ on all standard modules except those described in steps 1.--5. above.
This means that we have one of the following two cases:

\noindent{a)} $N{''}$ is conjugate to $G$. In this case $N_1{''}=N{''}=N'=G$, $\tau'$ is tempered irreducible representation of $G$, and
$i_{G, N'}\left(\chi|_{N'}\chi' \tau'\right)=\chi \chi' \tau'$. Thus, by the construction 1.--4., $\chi\longmapsto  f(\chi \chi' \tau')$ is regular.

\noindent{b)} $N{''}$  conjugate to $M$. In this case, $N'=N{''}$, and  $\tau{'}$ must be conjugate to an element of the orbit $\Psi^u(M)\sigma$ (see 5. above).
The discussion in 5. implies that  $\chi\longmapsto  f(i_{G, N'}\left(\chi|_{N'}\chi' \tau'\right))$ is regular.

This finally verifies (i) of (\cite{bdk}, 1.2), and completes the proof of the proposition.

\end{proof}

\end{document}